\documentclass{amsart}

\usepackage{amsmath,amscd,amssymb,enumerate,verbatim}

\newcommand{\Lie}{\mathfrak{L}}
\newcommand{\mJ}{\mathcal{J}}
\newcommand{\mK}{\mathcal{K}}
\newcommand{\mH}{\mathcal{H}}

\newcommand{\CC}{\mathbb{C}}
\newcommand{\RR}{\mathbb{R}}
\newcommand{\QQ}{\mathbb{Q}}

\newcommand{\ZZ}{\mathbb{Z}}
\newcommand{\End}{\mathrm{End}}
\newcommand{\id}{\mathrm{Id}}

\newcommand{\Symp}{\mathrm{Symp}}

\newcommand{\Nij}{\mathfrak{Nij}}

\newtheorem{thm}{Theorem}
\newtheorem{lma}{Lemma}

\newtheorem{qun}{Question}

\title{The infimum of the Nijenhuis energy}
\author{Jonathan David Evans}
\address{Department of Mathematics\\ ETH Z\"{u}rich\\ R\"{a}mistrasse 101\\ 8092 Z\"{u}rich\\ Switzerland}
\email{jonny.evans@math.ethz.ch}
\begin{document}
\begin{abstract}We prove that on any symplectic manifold whose symplectic form represents a rational cohomology class there exists a sequence of compatible almost complex structures whose Nijenhuis energy (the $L^2$-norm of the Nijenhuis tensor) tends to zero. The sequence is obtained by stretching the neck around a Donaldson hypersurface.\end{abstract}
\maketitle
\section{Introduction}
In this paper we examine the infimum of a certain functional on an infinite-dimensional space arising naturally in symplectic geometry. We will fix a smooth manifold $M$ of dimension $2n$ and on it a symplectic (i.e. closed, non-degenerate) 2-form $\omega$ and we consider the space $\mJ$ of \emph{compatible almost complex structures}, that is sections $J$ of the endomorphism bundle $\End(TM)$ such that
\begin{itemize}
\item $J^2=-\id$,
\item $\omega(X,JX)\geq 0$ with equality if and only if $X=0$,
\item $\omega(JX,JY)=\omega(X,Y)$.
\end{itemize}
The space $\mJ$ is contractible. Note that for each $J\in\mJ$ one obtains an almost K\"{a}hler metric $g_J(X,Y)=\omega(X,JY)$.

The Nijenhuis tensor of an almost complex structure $J$ takes two vector fields $V$ and $W$ and outputs a third
\[N_J(V,W)=[JV,JW]-[V,W]-J[V,JW]-J[JV,W]\]
The vanishing of the Nijenhuis tensor in a neighbourhood implies that local complex analytic coordinates can be found on that neighbourhood such that
\[J\partial_z=iz,\ J\partial_{\bar{z}}=-i\partial_{\bar{z}}\]
We call such a $J$ \emph{integrable} and write $\mK\subset\mJ$ for the locus of integrable compatible almost complex structures ($\mK$ stands for K\"{a}hler, since the metric $g_J$ is K\"{a}hler precisely when $J\in\mK$).

The functional we consider is
\[\Nij:\mJ\rightarrow\RR,\ \Nij(J)=\int_M|N_J|^2\omega^n\]
where the norm $|\cdot|$ is taken with respect to the almost K\"{a}hler metric $g_J$. Notice that this functional is invariant under the group $\Symp(M,\omega)$ of symplectomorphisms, that is diffeomorphisms which preserve the symplectic form. This functional was first considered by Blair and Ianus \cite{BI}; they were interested in the Goldberg conjecture which asserts that an almost K\"{a}hler metric which is Einstein is moreover K\"{a}hler and the functional is relevant because it has the same critical points as the restriction of the Einstein-Hilbert action to the locus of $\omega$-almost K\"{a}hler metrics, i.e. $\{g_J:J\in\mJ\}$.

In their paper \cite{LeWang} Le and Wang proved short-time existence for the downward gradient flow of this functional, which raises one's hope that one might be able to use variational methods or Ljusternik-Schnirelman theory to tackle questions about its critical points. A natural question (posed in \cite{LeWang}) is:
\begin{qun}
Does this functional ever have a positive infimum?
\end{qun}
If this were the case, the infimum would provide an interesting and entirely new invariant of non-K\"{a}hler symplectic manifolds. Le and Wang compute the gradient flow for an explicit (left-invariant) almost K\"{a}hler metric on the Kodaira-Thurston manifold (a four-dimensional symplectic non-K\"{a}hler nilmanifold) and show that for large flow-times the Nijenhuis energy approaches zero, so in that example the infimum of the Nijenhuis energy is zero. In this note, we prove:
\begin{thm}\label{nijinf}
If $(M,\omega)$ is a compact symplectic manifold such that $[\omega]\in H^2(M,\QQ)$ then the infimum of $\Nij$ over $\mJ$ is zero.
\end{thm}
This indicates that variational methods are unlikely to prove fruitful for the Nijenhuis energy functional: there are certainly many symplectic manifolds which admit no integrable compatible complex structure and yet one can make the Nijenhuis energy arbitrarily small.

Interestingly there are very few known non-K\"{a}hler critical points of $\Nij$ (i.e. critical points in $\mJ\setminus\mK$). The only compact examples know to the author are the Eells-Salamon almost complex structures on the twistor spaces of hyperbolic or complex-hyperbolic manifolds (equipped with the Reznikov symplectic form), see \cite{DavMus}.

Note that the proof does not work if we replace the $L^2$-norm of the Nijenhuis tensor by the $L^{2n}$-norm, i.e.
\[\int_M|N_J|^{2n}\omega^n\]
so it is still possible that this modified functional could have a non-zero infimum. Another possibility for obtaining a nonzero infimum is to replace $\mJ$ by the subset
\[\mJ_{CHSC}=\{J\in\mJ:\tilde{R}=\mathrm{const.}\}\]
where $\tilde{R}$ is the Hermitian scalar curvature and to infimise over this set. This seems natural since the symbol of the linearised Euler-Lagrange equation for the critical points of the Nijenhuis energy has degeneracies along the directions of the infinitesimal action of symplectomorphisms and along the directions of varying Hermitian scalar curvature (the author is indebted to Simon Donaldson for this idea).

A promising alternative to the downward gradient flow of the Nijenhuis energy was proposed in \cite{ST}, the \emph{symplectic curvature flow}. It would be intriguing to work out its relationship with the Nijenhuis energy.
\section{Proof}
We proceed by induction on dimension. For a symplectic 2-manifold $\mJ=\mK$ since any almost complex structure is integrable so there is nothing to prove. The induction is made possible by the following theorem of Donaldson:
\begin{thm}[Donaldson's symplectic submanifold theorem]
If $(M,\omega)$ is a symplectic manifold with $[\omega]\in H^2(M;\QQ)$ then there exists an integer $k>>0$ and a codimension 2 symplectic submanifold $\Sigma\subset M$ such that $P.D.[\Sigma]=k[\omega]$.
\end{thm}
\subsection{A neighbourhood of $\Sigma$}
The symplectic neighbourhood theorem tells us that, up to symplectomorphism, a neighbourhood of a symplectic submanifold of codimension 2 is determined by the first Chern class of its symplectic normal bundle. We therefore have the following normal form for $M$ in a neighbourhood of $\Sigma$ (stealing the notation of \cite{Bir}).
\begin{lma}
Let $\tau=\omega|_{\Sigma}$ and $\pi:\Lie\rightarrow\Sigma$ denote the symplectic normal bundle to $\Sigma$. Choose a fibrewise compatible complex structure $J^V$ and a Hermitian metric $|\cdot|$ on $\Lie$. Let $\nabla$ be a Hermitian connection on $\Lie$ with curvature 2-form $2\pi i\tau$ and denote its horizontal space by $\mH$. Write $X^{\mH}$ for the horizontal lift of a vector field $X$ on $\Sigma$. Denote by $r$ the fibrewise radial coordinate (as measured by $|\cdot|$). There is a unique 1-form $\alpha$ on $\Lie\setminus\Sigma$ such that
\begin{itemize}
\item $\mH\subset\ker\alpha$,
\item $\alpha(\partial_r)=0$,
\item $\alpha(J^V\partial_r)=\frac{1}{2\pi}$,
\item $d\alpha=-\pi^*\tau$.
\end{itemize}
Moreover, there exists an $0<\epsilon<1$ and a neighbourhood $\nu_{\epsilon}\Sigma$ of $\Sigma$ in $M$ which is symplectomorphic to the radius-$\epsilon$ disc subbundle $E_{\epsilon}$ in $\Lie$, where $E_{\epsilon}$ is equipped with the symplectic form
\[\Omega=\pi^*\tau+d(r^2\alpha).\]
\end{lma}
Note that in order to obtain a connection $\nabla$ with curvature $\tau$ we need $\tau\in H^2(\Sigma,\ZZ)$. This is not a problem since we have assumed $[\omega]\in H^2(M,\QQ)$ so some large scalar multiple is in $H^2(M,\ZZ)$ and the space of $\omega$-compatible almost complex structures is equal to the space of $k\omega$-compatible almost complex structures for any $k$.
\subsection{An almost complex structure near $\Sigma$}
If $\psi$ is a $\tau$-compatible almost complex structure on $\Sigma$ we construct an $\Omega$-compatible almost complex structure $J_{\psi}$ on $E_{\epsilon}$ by setting
\[
J_{\psi}X=\begin{cases}
(\psi Y)^{\mH} &\mbox{ if }X=Y^{\mH}\in\mH\\
J^V X&\mbox{ if }X\in T\pi^{-1}(p)
\end{cases}
\]
An easy computation yields:
\begin{lma}
$N_{J_{\psi}}(X,Y)=N_{\psi}(\pi_*X,\pi_*Y)^{\mH}$.
\end{lma}
\begin{proof}
Since the Nijenhuis tensor is $J$-invariant and vanishes if $Y=JX$ we need only consider two possibilities,
\begin{enumerate}
\item[(a)] $X^{\mH}\in\mH$, $Y=\partial_r$,
\item[(b)] $X^{\mH},Y^{\mH}\in\mH$.
\end{enumerate}
In case (a) we have
\[N_{J_{\psi}}(X^{\mH},\partial_r)=[(\psi X)^{\mH},J^V\partial_r]-[X^{\mH},\partial_r]-J_{\psi}[(\psi X)^{\mH} ,\partial_r]-J_{\psi}[X^{\mH},J^V\partial_r]\]
However, both $J_{\psi}$ and $X^{\mH}$ are invariant under the $\CC^*$-action generated by $\partial_r$ and $J^V\partial_r$ so all the Lie brackets vanish.

In case (b), at a point $(p,v)\in E_{\epsilon})$ we have
\begin{align*}
[X^{\mH},Y^{\mH}]&=[X,Y]^{\mH}+R_{\nabla}(X,Y)v\\
&=[X,Y]^{\mH}+2\pi \tau(X,Y)J^Vv
\end{align*}
Now since $\psi$ is $\tau$-compatible all the curvature terms cancel out when we take the linear combination of Lie brackets required for the Nijenhuis tensor and all that is left is
\[N_{J_{\psi}}(X^{\mH},Y^{\mH})=N_{\psi}(X,Y)^{\mH}\]
as required.
\end{proof}
To obtain an infimising sequence of compatible almost complex structures we will take a neighbourhood of $\Sigma$ equipped with such an almost complex structure and then apply a procedure called ``stretching the neck''. We recall this procedure from \cite{CM}.
\subsection{Stretching the neck}
We can decompose $M$ as $W\cup_Y\nu\Sigma$ where $W$ is the complement of $\nu\Sigma$ and $Y$ is the common boundary (the radius $\epsilon$ circle bundle in $\Lie$). We will take an almost complex structure $J_1\in\mJ$ which is of the form $J_{\psi}$ on $\nu\Sigma$ and extended in some arbitary way over $W$. Consider the map
\[\Phi:E_1\setminus\Sigma\rightarrow(-\infty,0)\times Y\]
defined by
\[\Phi(r,\theta,p)=(\ln(1-r^2),(\theta,p))\]
where $p\in\Sigma$, $(r,\theta)$ are polar coordinates on the fibre of $E_{\epsilon}$ and $\theta$ is also used for the fibre coordinate on $Y$. We also denote by $\Phi$ its restriction
\[\Phi:E_{\epsilon}\setminus\Sigma\rightarrow[\ln(1-\epsilon^2),0)\times Y\]
We have $\Phi^*d(e^t(-\alpha))=\Omega$. For any $K>0$ the neck-stretching procedure now replaces $M$ by
\[M_K=W\cup_Y[K\ln(1-\epsilon^2),0)\times Y\cup\Sigma\]
and the symplectic form by
\[\omega_K=\begin{cases}
e^{-K}\omega|_W&\mbox{ on }W\\
d(e^t(-\alpha))&\mbox{ on }[K\ln(1-\epsilon^2),0)\times Y\\
\omega&\mbox{ on }\nu\Sigma
\end{cases}\]
(Notice that we can still find $\nu\Sigma=[\ln(1-\epsilon^2),0)\times Y\cup\Sigma\subset M_K$). The stretched manifold $M_K$ is diffeomorphic to $M$ by shrinking the neck; under this diffeomorphism it is easy to check that the cohomology class $[\omega_K]$ is taken to $[\omega]$. Therefore Moser's argument implies that there exists a diffeomorphism $\Upsilon_K:M_K\rightarrow M$ such that $\Upsilon_K^*\omega=\omega_K$. We can define an almost complex structure $J'_K$ which agrees with $\Phi_*J_{\psi}$ on the neck and with $J_1$ on $W$. Since $J'_K$ is compatible with $\omega_K$, $(\Upsilon_K)_*J'_K$ is compatible with $\omega$. We call the sequence
\[J_K=(\Upsilon_K)_*J'_K\in\mJ\]
the neck-stretch of $J_1$.
\subsection{Completion of proof}
We compute the Nijenhuis energy of $J_K$:
\begin{lma}
As $K\rightarrow\infty$, $\Nij(J_K)\rightarrow\frac{2\pi n}{n+1}\Nij(\psi)$.
\end{lma}
\begin{proof}
Let $|\cdot|_K$ denote the norm on $(1,2)$-tensors associated with the metric $g_{J'_K}$.
\begin{align*}
\Nij(J_K)&=\Nij(J'_K)\\
&=\left(\int_W+\int_{[K\ln(1-\epsilon^2),0)\times Y}\right)|N_{J'_K}|_K^2\omega_K^n\\
&=e^{-Kn}\int_W|N_{J_1}|_K^2\omega^n+\int_{[K\ln(1-\epsilon^2),0)\times Y}|N_{\psi}|_K^2d(e^t(-\alpha))^n
\end{align*}
But on $W$, $|\cdot|_K=e^{2K}e^{-K}|\cdot|_1$ since the metric rescales by $e^{-K}$ and the tensor $N$ is of type $(1,2)$. Therefore the first term on the left is equal to
\[e^{-K(n-1)}\int_W|N_{J_1}|_1^2\omega^n\]
which tends to zero exponentially as $K\rightarrow\infty$.

The other term tends to the integral
\[\int_{(-\infty,0)\times Y}|N_{\psi}|_K^2d(e^t(-\alpha))^n=n\int_{(-\infty,0)}e^{nt}\left(\int_Y|N_{\psi}|_t^2 (-d\alpha)^{n-1}\wedge(-\alpha)\right)dt\]
Here $|\cdot|_t$ is the norm associated to $e^t(-d\alpha)=e^t\pi^*\tau$ and $\psi$ on $(1,2)$-tensors $\mH^*\otimes\mH^*\otimes\mH\rightarrow\RR$ and $Y$ is just a circle bundle over $\Sigma$ so the integral equals
\[2\pi n\int_{(-\infty,0)}e^{(n+1)t}dt\cdot\Nij(\psi)=\frac{2\pi n}{n+1}\Nij(\psi)\]
\end{proof}
Now by our inductive hypothesis we may pick a sequence of $\psi_i$ on $\Sigma$ such that $\Nij(\psi_i)\rightarrow 0$. We then construct a sequences $J_{i,K}$ by neck-stretching such that $\Nij(J_{i,K})\rightarrow\frac{2n\pi}{n+1}\Nij(J_i)$ as $K\rightarrow\infty$. Taking a diagonal subsequence $J_{i,K_i}$ we obtain a sequence of compatible almost complex structures whose Nijenhuis energy can be made arbitarily small, which proves Theorem \ref{nijinf}.

\section*{Acknowledgements}
The author would like to acknowledge useful and inspiring discussions with Simon Donaldson and Ivan Smith about the Nijenhuis energy. The referee's comments were also very helpful. This work was supported by an ETH Postdoctoral Fellowship.

\bibliographystyle{plain}\bibliography{neck-stretch-bib}
\end{document}